\documentclass{amsart}

\usepackage{enumerate,comment,tikz-cd}

\makeatletter
\newtheorem*{rep@theorem}{\rep@title}
\newcommand{\newreptheorem}[2]{%
\newenvironment{rep#1}[1]{%
 \def\rep@title{#2 \ref{##1}}%
 \begin{rep@theorem}}%
 {\end{rep@theorem}}}
\makeatother

\theoremstyle{plain}
\newtheorem{proposition}{Proposition}[section]

\newtheorem{lemma}{Lemma}[section]
\newtheorem{theorem}{Theorem}[section]
\newreptheorem{theorem}{Theorem}

\theoremstyle{definition}

\newtheorem{example}{Example}[section]

\theoremstyle{remark}
\newtheorem{remark}{Remark}[section]

\newcommand{\OO}{\mathcal{O}}

\newcommand{\DD}{{\mathcal D}}

\newcommand{\KK}{{\mathcal K}}

\newcommand{\de}{\delta}

\renewcommand{\ker}{\operatorname{ker}}
\newcommand{\ov}{\overline}

\renewcommand{\AA}{{\mathcal A}}

\newcommand{\FF}{{\mathcal F}}

\newcommand{\TT}{\mathcal{T}}
\newcommand{\XX}{{\mathcal X}}
\newcommand{\YY}{{\mathcal Y}}

\newcommand{\QQ}{{\mathcal Q}}
\newcommand{\VV}{{\mathcal V}}

\newcommand{\End}{\operatorname{End}}

\newcommand{\Aut}{\operatorname{Aut}}

\newcommand{\sub}{\subset}

\DeclareMathOperator{\Ext}{\mathrm{Ext}} 
 
\DeclareMathOperator{\Hom}{\mathrm{Hom}}

\DeclareMathOperator{\Spec}{\mathrm{Spec}}

\newcommand{\ot}{\otimes}
\newcommand{\id}{\operatorname{id}}

\begin{document}
\title[Bondal-Orlov-Bridgeland for DM Stacks]{Bondal-Orlov Fully
Faithfulness Criterion for Deligne-Mumford Stacks}

\author[B. Lim]{Bronson Lim}
\address{BL: Department of Mathematics \\ University of Utah \\ Salt Lake City,
  UT 84102, USA} 
\email{bcl@uoregon.edu}
\author[A. Polishchuk]{Alexander Polishchuk}
\address{AP: Department of Mathematics, 
    University of Oregon,     Eugene, OR 97403, USA; National Research University Higher
School of Economics; and Korea Institute for  Advanced Study} 
\email{apolish@uoregon.edu}


\begin{abstract}
  Suppose \(F\colon \mathcal{D}(X)\to \mathcal{T}\) is an exact functor from
  the bounded derived category of coherent sheaves on a smooth projective
  variety \(X\) to a triangulated category \(\mathcal{T}\).  If \(F\) possesses
  left and right adjoints, then the Bondal-Orlov criterion gives a simple way
  of determining if \(F\) is fully faithful. We prove a natural extension of this
  theorem to the case when $X$ is a smooth and proper DM stack with projective coarse
  moduli space.
\end{abstract}

\maketitle

\section{Introduction}
\label{sec:intro}

\subsection{Bondal-Orlov Criterion}

Suppose \(X\) is a smooth projective scheme over an algebraically closed field \(k\)
of characteristic zero and \(F\colon \mathcal{D}(X)\to \mathcal{T}\) is
an exact functor, with \(\mathcal{T}\) a triangulated category. One is often
interested in checking whether \(F\) embeds \(\mathcal{D}(X)\) as a full triangulated
subcategory of \(\mathcal{T}\). If \(F\) admits left and right adjoints, then the
following well-known Bondal-Orlov Criterion is the primary tool used,
\cite{bondal-orlov-sod,bridgeland-triangulated-99}.

\begin{theorem}
  The functor \(F\) is fully faithful if, and only if, it admits a right adjoint, a left adjoint $G$, with $G\circ F$ of Fourier-Mukai type,
  and 
  \begin{itemize}
  \item
  for any closed point
  \(x\in X\), one has
  \[
    \mathrm{Hom}_{\mathcal{T}}(F(\mathcal{O}_x),F(\mathcal{O}_x)) = k;
  \]
 \item
 for any pair of closed points \(x,y\in X\) one has
  \[
    \mathrm{Hom}_{\mathcal{T}}(F(\mathcal{O}_x),F(\mathcal{O}_y)[i])=
    0\text{ unless }x = y\text{ and }0\leq i\leq \dim(X).
  \]
  \end{itemize}
  \end{theorem}

This theorem has been extended to the quasi-projective \cite{martin-ffqproj} and gerby projective 
\cite{cal-k3} setting, to
the case when $X$ is allowed to have some singularities, and to the case of positive characteristic 
\cite{rms-ffpositivechar,salas-ff-09}.

Recent interest in derived categories of Deligne-Mumford stacks (see e.g., \cite{bergh-lunts-schnurer,hall-rydh-perfect}) warrants an
investigation of a similar criterion for this category. In this article, we
extend the Bondal-Orlov criterion to the class of smooth and proper Deligne-Mumford
stacks with projective coarse moduli.

In the case of stacks the notion of a $k$-point has to be replaced by that of a
{\it generalized point} $(x,\xi)$, where $x:\Spec(k)\to \XX$ is a morphism and
$\xi$ is an irreducible representation of $\Aut(x)$, see Proposition \ref{prop:spanning-pts}.  These pairs are considered
up to an isomorphism. For each generalized point, $(x,\xi)$ there is a natural coherent sheaf
$\OO_{x,\xi}$ on $\XX$, which is an analog of the skyscraper sheaf (see Sec.\ \ref{points-sec}).

\begin{theorem}\label{thm:bob-stacks} Let $\XX$ be a smooth and proper DM-stack
  wtih projective coarse moduli space over an algebraically closed field $k$ of
  characteristic zero.  Suppose \(F\colon \mathcal{D(X)\to T}\) is an exact
  functor with a right adjoint and a left adjoint \(G\colon \mathcal{T\to D(X)}\) such that
  \(G\circ F\) is of Fourier-Mukai type.  Then \(F\) is fully-faithful if and
  only if 
  \begin{itemize}
    \item  for each generalized point $(x,\xi)$ of \(\mathcal{X}\), one has
      \[
        \mathrm{Hom}_{\mathcal{T}}(F(\mathcal{O}_{x,\xi}),
        F(\mathcal{O}_{x,\xi})) = k;
      \]
    \item for each pair of generalized points $x,y$, one has
      \[
        \mathrm{Hom}_{\mathcal{T}}(F(\mathcal{O}_{x,\xi}),
        F(\mathcal{O}_{y,\eta})[i])=0\text{ unless } x\simeq y \text{ and }0\leq
        i\leq \dim(\mathcal{X}); \text{ and }
      \]
      \[
        \mathrm{Hom}_{\mathcal{T}}(F(\mathcal{O}_{x,\xi}),
        F(\mathcal{O}_{y,\eta}))=0\text{ unless } (x,\xi)\simeq (y,\eta).
      \]
  \end{itemize}
\end{theorem}

\begin{remark}
 Note that the natural dg-enhancement of $\DD(\XX)$ is saturated. 
 Thus, in the case when the triangulated category \(\mathcal{T}\) admits a dg enhancement and
  \(F\) lifts to the dg level, the conditions that \(F\) admits left and right 
  adjoints and that \(G\circ F\) is of Fourier-Mukai type are automatic
 (see \cite[Theorem 1.3]{genovese-adjunctions}, \cite[Theorem 1.2]{bzfn-integraltransforms}). 
\end{remark}

\subsection{Outline of the paper}

The proof will proceed similarly to the proof of the original Bondal-Orlov criterion in \cite[Section 7.1]{huybrechts-fm}. 
We collect the relevant background material in Section \ref{sec:prelims}. 
The key technical idea is to use the trade-off between nontrivial generic stabilizer and ``gerbyness".
This was observed by Bergh-Gorchinskiy-Larsen-Lunts in \cite{bgll-17} in the form of an equivalence of
the category of \(G\)-equivariant coherent sheaves corresponding to an ineffective action of a finite group $G$,
with some ``gerby" category. We call this the {\it BGLL equivalence} and recall the details in Section \ref{sec:bgll}.
We complete the proof in Section \ref{sec:main-thm}.

\bigskip

\noindent
{\it Conventions}.
Throughout \(k\) will be an algebraically closed field of characteristic zero.
Unless otherwise stated, our stacks will be Deligne-Mumford, smooth and proper over \(k\) with
projective coarse moduli. Note that by \cite[Thm.\ 4.4]{kresch-geometrydm}, any such stack is a quotient stack of a 
quasi-projective scheme by a linear algebraic group.
All functors are assumed to be derived. The bounded
derived category of coherent sheaves on \(\mathcal{X}\) is denoted by
\(\mathcal{D(X)}\). 

\bigskip

\noindent
{\it Acknowledgments}.
The authors are grateful to the anonymous referee for helpful comments and suggestions.
A.P. is supported in part by the NSF grant DMS-1700642 and by the Russian Academic Excellence Project `5-100'.
He is grateful to Jarod Alper for an enlightening discussion. 

\section{Preliminaries on DM Stacks and Triangulated Categories}
\label{sec:prelims}

\subsection{Serre Duality}

Since our stacks are smooth and proper, the exotic inverse image functor
\[
  p^!\colon\mathcal{D}(\mathrm{Spec}(k))\to \mathcal{D(X)}
\] 
is defined, see \cite{nironi-08}, and we set \(\omega_{\mathcal{X}} =
p^!\mathcal{O}_{\mathrm{Spec}(k)}\) to be the dualizing sheaf on
\(\mathcal{X}\). Moreover, the associated endofunctor \(S\colon\mathcal{D(X)\to
D(X)}\) given by
\[
  S(\mathcal{F}) = (\mathcal{F}\otimes\omega_{\mathcal{X}})[\dim(\mathcal{X})]
\]
is a Serre functor for \(\mathcal{D(X)}\).

\subsection{Points}\label{points-sec}

By a closed point of \(\mathcal{X}\), we mean a morphism
\(x\colon\mathrm{Spec}(k)\to \mathcal{X}\). Any closed point factorises over the inclusion of
{\it residual gerbe} as
\[\mathrm{Spec}(k)\xrightarrow{p}B\mathrm{Aut}(x)\xrightarrow{\iota_x}\mathcal{X}\] 
Here, \(\mathrm{Aut}(x)\) is the finite stabilizer group of \(x\), \(B\mathrm{Aut}(x) \cong
[\mathrm{pt}/\mathrm{Aut}(x)]\) is the classifying stack, and $p$ is its canonical \'etale atlas. 

For any finite group \(G\), Maschke's Theorem gives a completely orthogonal decomposition
\begin{equation}
    \mathcal{D}(BG) = \bigoplus\limits_{\xi\in \mathrm{Irr}(G)}
    \mathcal{D}(\mathrm{Spec}(k))\otimes\xi.
\label{eqn:sod-bg}
\end{equation}
For any closed point, \(x\colon\mathrm{Spec}(k)\to \mathcal{X}\), and
irreducible representation, \(\xi\in\mathrm{Irr(Aut}(x))\), we denote
by \(\mathcal{O}_{x,\xi}\) the sheaf
\(\iota_{x\ast}(\mathcal{O}_{\mathrm{Spec}(k)}\otimes\xi)\). We will think of
the pair \( (x,\xi)\) as a \textit{generalized point} with structure sheaf
\(\mathcal{O}_{x,\xi}\).

Note that $p_*\OO_{\mathrm{Spec}(k)}$ is the regular $\mathrm{Aut(x)}$-representation, hence, $\OO_x:=x_*\OO_{\mathrm{Spec}(k)}$
has the following decomposition:
\begin{equation}\label{Ox-decomposition}
\OO_x\simeq \bigoplus_{\xi\in\mathrm{Irr(Aut}(x))} \xi^\vee\otimes \OO_{x,\xi}.
\end{equation}

\subsection{Fourier-Mukai Functors}

Let \(\mathcal{X,Y}\) be smooth and proper DM stacks of finite
type over \(k\) with generically trivial stabilizers. Any object
\(\mathcal{P}\in \mathcal{D(X\times Y)}\), determines an exact functor
\[
  \Phi_{\mathcal{P}}\colon \mathcal{D(X)\to D(Y)}
\]
defined by the formula
\[
  \Phi_{\mathcal{P}}(\mathcal{E}) =
  \pi_{\mathcal{Y}\ast}(\pi_{\mathcal{X}}^\ast(\mathcal{E}) \otimes
  \mathcal{P}).
\]
We will say that an exact functor \(F\colon \mathcal{D(X)\to D(Y)}\) is of
\textit{Fourier-Mukai type} or an \textit{integral functor} if \(F\cong
\Phi_{\mathcal{P}}\) for some \(\mathcal{P}\in\mathcal{D(X\times
Y)}\). Since Serre duality holds in this setting, such functors have left and right adjoints given by the usual formulas
(see \cite[Prop.\ 5.9]{huybrechts-fm}). In particular, the criterion of Theorem \ref{thm:bob-stacks} is applicable to functors of Fourier-Mukai type.

\begin{example}
  Let \(\mathcal{X}\) be a DM stack, then the diagonal object
  \(\Delta_\ast\mathcal{O}_{\mathcal{X}}\in\mathcal{D}(\mathcal{X\times X})\) is
  a kernel for the identity functor
  \[
    \Phi_{\Delta_\ast\mathcal{O}_{\mathcal{X}}}\cong \mathrm{Id}\colon
    \mathcal{D(X)\to D(X)}.
  \]
  Note that \(\Delta\colon\mathcal{X\to X\times X}\) is finite and so the
  argument in \cite[Example 5.4(i)]{huybrechts-fm} carries over exactly.
  \label{ex:identity}
\end{example}

\subsection{Spanning classes}

Recall that a \textit{spanning class} in a triangulated category \(\mathcal{T}\) is a
subclass of objects \(\Omega\subset \mathcal{T}\) such that for all \(t\in
\mathcal{T}\) we have:
\begin{align*}
  & \mathrm{Hom}_{\mathcal{T}}(\omega[i],t)=0\text{ for all
  }i\in\mathbb{Z}\text{ and for all }\omega\in\Omega\text{ implies }t=0; \\
  & \mathrm{Hom}_{\mathcal{T}}(t,\omega[i])=0\text{ for all
  }i\in\mathbb{Z}\text{ and for all }\omega\in\Omega\text{ implies }t=0.
\end{align*}

In the case of (quasi-)projective varieties, it is well known that the structure sheaves of closed points form a
spanning class. We need an analogue in the stacky setting.
The following proposition seems to be well known and follows analagously to
\cite{huybrechts-fm}. We include the proof for
completeness and for lack of a suitable reference.

\begin{proposition}
  The subclass of objects 
  \[
    \Omega_{pt} = \{\mathcal{O}_{x,\xi}\mid x\colon\mathrm{Spec}(k)\to
    X\text{ and }\xi\in\mathrm{Irr}(\mathrm{Aut}(x))\}
  \]
  form a spanning class in \(\mathcal{D(X)}\).
  \label{prop:spanning-pts}
\end{proposition}

\begin{proof}
  By Serre duality, it suffices to show that if
  \(\mathcal{F}\in\mathcal{D(X)}\) is not zero, then there exists a
  \(\mathcal{O}_{x,\xi}\) and \(i\in\mathbb{Z}\) such that
  \[
    \mathrm{Hom}(\mathcal{F},\mathcal{O}_{x,\xi}[i])\neq 0
  \]
  Since \(\mathcal{F}\neq 0\) and is bounded, there exists a maximal
  \(m\) such that \(m\)th cohomology sheaf \(\mathcal{H}^m:=\mathcal{H}^m(\mathcal{F})\) is nonzero. Now
  using the spectral sequence
  \[
    E_2^{p,q} = \mathrm{Hom}(\mathcal{H}^{-q},\mathcal{O}_{x,\xi}[p])\Rightarrow
    \mathrm{Hom}(\mathcal{F},\mathcal{O}_{x,\xi}[p+q])
  \]
  we see that the differentials with source \(E_r^{0,-m}\) are zero for all
  \(r\geq 2\) and, similarly to the non-stacky case, all the differentials with
  target \(E_r^{0,-m}\) are also trivial. Thus, \(E_\infty^{0,-m} =
  E_2^{0,-m}\). Since \(\mathcal{H}^{m}\) is a sheaf,
  there exists a residual gerbe \(\iota_x\colon
  B\mathrm{Aut}(x)\to \mathcal{X}\) such that \(\iota_x^\ast\mathcal{H}^m\neq
  0\). 

  Since \(\iota_x^\ast\mathcal{H}^m\neq 0\), there exists an irreducible
  representation \(\xi\) and a nonzero morphism \(\mathcal{H}^m\to
  \mathcal{O}_{\mathrm{Spec}(k)}\otimes\xi\). Since
  \[
    E_\infty^{0,-m} =
    E_2^{0,-m}=\mathrm{Hom}(\mathcal{H}^m,\mathcal{O}_{x,\xi})\neq 0,
  \]
  we conclude \(\mathrm{Hom}(\mathcal{F},\mathcal{O}_{x,\xi}[-m])\neq 0\)
  as desired.

\end{proof}

\begin{remark}
  The spanning class in Proposition \ref{prop:spanning-pts} should be thought of
  as a refinement of the spanning class
  \[
    \Omega = \{ \mathcal{O_Z}\mid \mathcal{Z}\text{ is a closed substack of
    }\mathcal{X}\text{ and }\pi(\mathcal{Z})\text{ is a closed point in }X\}
  \]
  where \(\pi\colon\mathcal{X}\to X\) is the coarse moduli, see
  \cite{CT-derivedmckay}.
\end{remark}

Recall that spanning classes can be used to check fully-faithfulness of
exact functors.

\begin{proposition}[\protect{\cite[Thm.\ 2.3]{bridgeland-triangulated-99}}]
  Suppose \(F\colon \mathcal{T\to T'}\) is an exact functor with left and
  right adjoints. Then \(F\) is fully-faithful if and only if there exists a
  spanning class \(\Omega\subset\mathcal{T}\) such that for all
  \(\omega,\omega'\in\Omega\) the induced map
  \[
    \Hom_{\TT}(\omega,\omega'[i])\to \Hom_{\TT'}(F(\omega),F(\omega')[i])
  \]
  is an isomorphism for all \(i\).
  \label{spanning-ff-prop}
\end{proposition}

\subsection{Some Lemmas}

We need the following criterion for a complex to be a sheaf, flat over the base, as in
\cite{bridgeland-triangulated-99}.

\begin{lemma}
  Let \(\pi\colon \mathcal{S\to T}\) be a morphism of DM stacks, and for each closed
  point \(t\colon\mathrm{Spec}(k)\to \mathcal{T}\), let \(j_t\colon
  \mathcal{S}_t\to \mathcal{S}\) denote the inclusion of the fiber
  \(\mathcal{S}_t= \mathcal{S}\times_\mathcal{T}\mathrm{Spec}(k)\). Let
  \(\mathcal{Q}\) be an object of \(\mathcal{D(S)}\) such that for all \(t\colon
  \mathrm{Spec}(k)\to \mathcal{T}\), the derived restriction
  \(j_t^\ast(\mathcal{Q})\) is a sheaf on \(\mathcal{S}_t\). Then
  \(\mathcal{Q}\) is a sheaf on \(\mathcal{S}\), flat over \(\mathcal{T}\).
  \label{lem:flatness-criterion}
\end{lemma}

\begin{proof}
  We remark that \(\mathcal{Q}\) is a sheaf, flat over \(\mathcal{T}\), if and
  only if the base change to an \'etale cover on the source and on the target is
  a sheaf, flat over the base. Using this, we can deduce the assertion of our lemma from \cite[Lemma
  4.3]{bridgeland-triangulated-99} as follows. 

  Specifically, pick an \'etale cover \(p_T\colon T\to
  \mathcal{T}\) by a scheme $T$. Then the morphism \(t\colon \mathrm{Spec}(k)\to \mathcal{T}\)
  lifts to \(t'\colon\mathrm{Spec}(k)\to T\).\footnote{We are using that \(k\)
  is algebraically closed here.} Set \(\mathcal{S}_T =
  \mathcal{S}\times_\mathcal{T}T\) and so for any \(t\in\mathcal{T}(k)\), we can
  set \(\mathcal{S}_t \cong \mathcal{S}_T\times_T\mathrm{Spec}(k)\). Let
  \(p_S\colon S_T\to \mathcal{S}_T\) be an \'etale cover of \(\mathcal{S}_T\) and
  \(S_t=S_T\times_T\mathrm{Spec}(k)\). We have the following diagram where
  all squares are Cartesian
  \[
    \begin{tikzcd}
      S_t\ar{d}\ar{r} & S_T\ar{d}{p_S} & {} \\
      \mathcal{S}_t\ar{r}{s}\ar{d}{\pi''} &
      \mathcal{S}_T\ar{r}{p_T'}\ar{d}{\pi'} & \mathcal{S}\ar{d}{\pi} \\
      \mathrm{Spec}(k)\ar{r}{t'} & T\ar{r}{p_T} & \mathcal{T}
    \end{tikzcd}
  \]
  Thus, \(\mathcal{Q}\) is a sheaf, flat over \(\mathcal{T}\) if and only if
  \(\mathcal{Q}' = (p_T'\circ p_S)^\ast\mathcal{Q}\) is a sheaf, flat over
  \(T\). The statement now follows from \textit{loc. cit.}
\end{proof}


\begin{lemma}
  Let $x\in \XX(k)$ be a point, and
  \(\mathcal{F}\in \mathcal{D(X)}\). Suppose
  \[
    \mathrm{Hom}(\mathcal{F},\mathcal{O}_{y,\eta}[i]) = 0
  \]
  for \(i\in\mathbb{Z}\), all points $y\neq x$, and all $\eta\in  \mathrm{Irr}(\mathrm{Aut}(y))$, and 
  \[
    \mathrm{Hom}(\mathcal{F},\mathcal{O}_{x,\xi}[i]) = 0
  \]
  for \(i\notin[0,\dim(\mathcal{X})]\) and all $\xi\in  \mathrm{Irr}(\mathrm{Aut}(x))$.

  Then \(\mathcal{F}\) is a sheaf supported at $x$. 
  \label{lem:sheaf-lemma}
\end{lemma}

\begin{proof}
  Let \(\pi\colon U\to \mathcal{X}\) be an \'etale cover by a scheme $U$. If
  \(\pi^\ast\mathcal{F}\) is a sheaf concentrated at \(\pi^{-1}(x)\), then
  \(\mathcal{F}\) must also be a sheaf concentrated at \(x\). But by the
  argument in \cite[Lemma 7.2]{huybrechts-fm} and our assumptions, the object
  \(\pi^\ast\mathcal{F}\) is a sheaf concentrated at \(\pi^{-1}(x)\).
\end{proof}

\section{Ineffective group actions and twisted sheaves}
\label{sec:bgll}

We will use the following description of the derived category \(\mathcal{D}[X/G]\) of a global quotient 
by a non-effective action of a finite group from \cite[Theorem 5.5(i)]{bgll-17} in terms of
sheaves twisted by a Brauer class. 

\subsection{BGLL Equivalence with twisted sheaves}

Suppose \(G\) is a finite group and \(X\) is a smooth quasi-projective
\(G\)-variety. Let us denote by \(N\sub G\) the kernel of the action so that \(H = G/N\) acts
effectively on \(X\). In \cite{bgll-17}, the authors describe the category
\(\mathfrak{Coh}[X/G]\) in terms of twisted \(H\)-equivariant sheaves on
\(\mathrm{Irr}(N)\times X\). We recall this now.

Let \(V\) be any representation of \(G\) and consider the algebra
\[
A: = \mathrm{End}_N(V)^{op}
\]
We will assume that \(V\) is an
\(N\)-generator, i.e., \(V\) contains all irreducible representations of \(N\).
For example, \(V = k[G]\) would work.

Let \(Z\) be the center of the group algebra of \(N\),
\[
  Z := Z(k[N]),
\]
and let
\[
  \mathrm{Irr}(N):= \Spec(Z)
\]
denote the scheme of irreducible representations (discrete under our assumptions). 
The group \(H\) acts naturally on
\(\mathrm{Irr}(N)\), and \(A\) is an \(H\)-equivariant
Azumaya algebra over \(\mathrm{Irr}(N)\) (via the natural embedding $Z\to A$). Hence, \(A\) determines an \(H\)-equivariant Brauer
class \(\alpha\in \mathrm{Br}^H(\mathrm{Irr}(N))\). 
More precisely, if 
$$V=\bigoplus_{\xi\in \mathrm{Irr}(N)}V_\xi \ot \xi,$$
is a decomposition of $V$ into $N$-isotypic components, then
\begin{equation}\label{A-dec-eq}
A=\bigoplus_{\xi\in \mathrm{Irr}(N)} \mathrm{End}_k(V_\xi)^{op}.
\end{equation}

We equip \(\mathrm{Irr}(N)\times X\) with the diagonal \(H\)-action, and denote by
\(\pi_1\colon \mathrm{Irr}(N)\times X\to \mathrm{Irr}(N)\) and \(\pi_2\colon
\mathrm{Irr}(N)\times X\to X\) 
the natural ($H$-equivariant) projections. Let us consider the sheaf of algebras  
\[\mathcal{A} := A\otimes_k\mathcal{O}_X\simeq \pi_{2\ast}(\pi_1^\ast A)\]
on $X$, equipped with an $H$-equivariant structure.

Since \(\pi_2\) is a
finite morphism, we have an equivalence of categories
\[
  \pi_{2\ast}\colon
   \mathfrak{Coh}^H(\mathrm{Irr}(N)\times X,\pi_1^\ast \alpha)=
  \mathfrak{Coh}^H(\mathrm{Irr}(N)\times X,\pi_1^\ast A)\xrightarrow{\sim}
  \mathfrak{Coh}^H(X,\AA).
\]

Set
\(\mathcal{V}:=V\otimes_k\mathcal{O}_X\) and consider the functor
\begin{equation}\label{Hom-V-functor}
  \mathcal{H}om_N(\mathcal{V},-)\colon\mathfrak{Coh}^G(X)\to
  \mathfrak{Coh}^H(X,\mathcal{A}).
\end{equation}

\begin{theorem}[BGLL Equivalence]
  There is an equivalence of categories
  \[
    \mathfrak{Coh}^G(X)\simeq \mathfrak{Coh}^H(\mathrm{Irr}(N)\times X, \pi_1^\ast\alpha).
  \]
  given by \(\pi_{2\ast}^{-1}\circ\mathcal{H}om_N(\mathcal{V},-)\).
  \label{thm:bgll}
\end{theorem}

Let us consider the stack quotient 
\[\bar{X}_N = [(\mathrm{Irr}(N)\times X)/H]\]
which has trivial generic automorphism group.
The $H$-equivariant class \(\pi_1^\ast\alpha\) defines an element \(\bar{\alpha}\) in the Brauer group
\(\mathrm{Br}(\bar{X}_N)\), so we can rewrite the above equivalence as 
\[\mathfrak{Coh}^G(X)\simeq \mathfrak{Coh}(\bar{X}_N,\bar{\alpha}).\]  
We can represent $\bar{\alpha}$ by the Azumaya algebra $\mathcal{A}'$, the descent of $\pi_1^\ast A$ to $\bar{X}_N$. 
Note that the push-forward of $\mathcal{A}'$ under the projection $\pi_{2*}:\bar{X}_N\to [X/H]$ is precisely $\mathcal{A}$ viewed
as a sheaf of algebras on $[X/H]$.
Note also that if we consider $\VV$ as a vector bundle on the stack $[X/G]$ then the functor \eqref{Hom-V-functor}
can be identified with the functor
$$\mathfrak{Coh}([X/G])\to\mathfrak{Coh}([X/H],\mathcal{A})\colon\FF\mapsto \pi_*(\mathcal{H}om_{[X/G]}(\VV,\FF)),$$
where $\pi:[X/G]\to [X/H]$ is the natural morphism.

\subsection{BGLL Equivalence and generalized points}

Let us assume in addition that $H$ acts freely on $X$, so that $\bar{X}_N$ is a usual scheme (not a stack).

For each generalized point \( (x,\xi)\) of $[X/G]$, we set \(\overline{(x,\xi)}\) to be the image of $(x,\xi)$ 
in \(\bar{X}_N\) which is a \(k\)-point of $\bar{X}_N$. The
corresponding skyscraper sheaf \(\mathcal{O}_{\overline{(\xi,x)}}\) can be viewed as an 
\(\bar{\alpha}\)-twisted sheaf on \(\bar{X}_N\). More precisely, we use the natural splitting of $A$ over $\xi$
(see decomposition \eqref{A-dec-eq}),
so the corresponding sheaf of $\mathcal{A}'$-modules is $V_\xi^{\vee}\otimes\mathcal{O}_{\overline{(\xi,x)}}$.

\begin{lemma}\label{BGLL-points-lem}
    Under the BGLL equivalence above, the structure sheaves of generalized points
    \(\mathcal{O}_{x,\xi}\) are mapped to the skyscraper sheaves
    \(\mathcal{O}_{\overline{(\xi,x)}}\) viewed as twisted sheaves.
\end{lemma}

\begin{proof}
  Recall that for a generalized point $(x,\xi)$, one has \(\mathcal{O}_{x,\xi} = \iota_{x\ast}(\xi)\). We have a commutative diagram
\[  
    \begin{tikzcd}
      \mathrm{pt}/N\ar{d}\ar{r}{\iota_x} & {[X/G]}\ar{d}{\pi} \\
      \mathrm{pt} \ar{r}{\iota_{\overline{x}}} & {[X/H]}
    \end{tikzcd}
\]
where $\overline{x}$ is the induced point of $X/H$.
  Thus,
  \[
    \pi_*\mathcal{H}om_{[X/G]}(\mathcal{V},\mathcal{O}_{x,\xi}) \cong
    \pi_*\iota_{x\ast}\mathrm{Hom}(V,\xi)\cong 
    \iota_{\ov{x}\ast}V_\xi^\vee,
  \]
which is viewed as an $\mathcal{A}$-module via the projection $A\to \End(V_\xi)^{op}$.
On the other hand, as we have seen above, the structure sheaf of the point $\overline{(x,\xi)}\in \bar{X}_N$
corresponds to the $\mathcal{A}'$-module $V_\xi^{\vee}\otimes\mathcal{O}_{\overline{(\xi,x)}}$.
Thus, its push-forward under $\pi_{2*}$ is isomorphic to the $\mathcal{A}$-module $\iota_{\ov{x}\ast}V_\xi^\vee$.
\end{proof}

\subsection{BGLL Equivalence and Fourier-Mukai Functors}

Let \(\QQ\) be a \(G\times G\)-equivariant sheaf on \(X\times X\). Then \(\QQ\)
determines, under the BGLL equivalence, a twisted sheaf \(\QQ'\) on
\(\bar{X}_N\times \bar{X}_N\).

\begin{lemma}
  Suppose \(\QQ\) is flat over \(X\) via the first projection, then \(\QQ'\) is
  flat over \(\bar{X}_N\) over the first projection.
  \label{lem:flatness-bgll}
\end{lemma}

\begin{proof}
  We just need to check that the functor \(\mathcal{H}om_N(\mathcal{V},-)\)
  preserves flatness as all of the other functors clearly do. But this is clear
  as \(\mathcal{V}\) is a vector bundle.
\end{proof}

The last ingredient is a generalization of Bridgeland's Hilbert scheme argument
(see the proof of \cite[Lem.\ 5.3]{bridgeland-triangulated-99}).

For a smooth quasiprojective scheme $S$, we denote by $\mathrm{Hilb}_{\ell}(S)$
 the Hilbert scheme of length $\ell$ finite subschemes.

Let \( (Y,\alpha)\) be a twisted smooth scheme and \(\pi\colon U\to Y\) an \'etale cover by a scheme $U$
trivializing \(\alpha\). 

\begin{lemma}
  Suppose \(\QQ\) is a coherent \(\pi_2^\ast\alpha\)-sheaf on \( U\times Y\), for
  \(\alpha\in\mathrm{Br}(Y)\), which is flat over \(U\). Suppose for each closed
  point \(u\in U\), the following two conditions hold:
  \begin{itemize}
    \item \(\QQ_u:= \QQ_{\{u\}\times Y}\) is concentrated at \(\pi(u)\);
    \item \(\mathrm{Hom}(\QQ_u,\mathcal{O}_{\pi(u)}) = k\).
  \end{itemize}
  Then there exists a non-empty open subsheme \(U'\) of \(U\) such that the corresponding
  composite map
  \[
    U'\to \mathcal{C}oh(Y,\alpha)\xrightarrow{\pi^\ast}\mathcal{C}oh(U)
  \]
  factors through a finite map \(f:U'\to \mathrm{Hilb}_\ell(U)\), for some $\ell\ge 0$, where \(\mathcal{C}oh(Y,\alpha)\) is the
  stack of coherent \( (Y,\alpha)\)-twisted sheaves. In other words, there is an isomorphism of coherent sheaves
  over $U'\times U$,
  $$(\id_{U'}\times\pi)^*\QQ\simeq (f\times\id_U)\mathcal{O}_{\Xi},$$
  where $\Xi\subset \mathrm{Hilb}_\ell(U)\times U$ is the universal family of length $\ell$ subschemes.
  \label{lem:bridgeland-hilbert-argument}
\end{lemma}

\begin{proof}
  Since \(\QQ_u\) is concentrated at \(\pi(u)\), its scheme-theoretic support is a
  zero-dimensional subscheme of \(Y\). 
 Since the restriction of the class $\alpha$ to the support of $\QQ_u$ is trivial, we can view \(\QQ_u\) as an honest sheaf. Bridgeland's
  original argument shows that \(\QQ_u\) is the structure sheaf of a
  zero-dimensional subscheme.

  Let \(\QQ'\) denote the induced family on \(U\times U\). That is, \(\QQ'\) is
  the pullback of \(\QQ\). Then for each \(u\in U\), \(\QQ'_u
  :=\QQ'|_{\{u\}\times U}\) is the structure sheaf of a zero-dimensional
  subscheme of \(U\). The local map \(\mathcal{O}_U\to
  \QQ'_u\) extends  to a section \(H^0(U\times U,\QQ')\) which is surjective over $U'\times U$,
for a smaller open set $U'\subset U$. 
  This means that $\QQ'|_{U'\times U}\simeq \mathcal{O}_{\mathcal{Z}}$ for some subscheme $\mathcal{Z}\subset U'\times U$
  which is flat and finite over $U'$.
  Then we have the commutative diagram:
  \[
  \begin{tikzcd}
    U'\ar{r}{\QQ} \ar{dr}{\QQ'} & Coh(Y,\alpha) \ar{r}{\pi^\ast}& Coh(U) \\
    {} & \mathrm{Hilb}_{\ell}(U)\ar{ur}
  \end{tikzcd}.
  \]
\end{proof}

\section{Proof of Theorem \ref{thm:bob-stacks}}
\label{sec:main-thm}

We will proceed similarly to \cite[Section 7.1]{huybrechts-fm}. We have
already shown  in Proposition
\ref{prop:spanning-pts} that generalized points $\OO_{x,\xi}$ are spanning. Thus, by Proposition \ref{spanning-ff-prop}, we just need to show that the natural homomorphisms
\begin{equation}
  \mathrm{Hom}_{\mathcal{D}(X)}(\mathcal{O}_{x,\xi},\mathcal{O}_{y,\zeta}[i])\to
  \mathrm{Hom}_{\mathcal{T}}(F(\mathcal{O}_{x,\xi}),F(\mathcal{O}_{y,\zeta})[i])
  \label{eqn:bijectivity}
\end{equation}
are isomorphisms for all generalized points
\(\mathcal{O}_{x,\xi},\mathcal{O}_{y,\zeta}\) and any integer
\(i\in\mathbb{Z}\). The proof will occupy the remainder of this section.

\subsection{Reduction to \(G(F(\mathcal{O}_{x,\xi}))\cong \mathcal{O}_{x,\xi}\)}

As in the original proof, to prove 
(\ref{eqn:bijectivity}), we have to show the bijectivity of the map
\[
  \mathrm{Hom}_{\mathcal{D}(X)}(\mathcal{O}_{x,\xi},\mathcal{O}_{y,\zeta}[i])\to
  \mathrm{Hom}_{\mathcal{T}}(GF(\mathcal{O}_{x,\xi}),\mathcal{O}_{y,\zeta}[i])
\]
induced by the adjunction morphism \(G\circ F\to
\mathrm{Id}_{\mathcal{D(X)}}\), where $G$ is the left adjoint functor to $F$.

If \(GF(\mathcal{O}_{x,\xi})\cong \mathcal{O}_{x,\xi}\), then either the
adjunction morphism is zero or it is an isomorphism. But as in the original
proof, it cannot be zero as
\[
  \mathrm{Hom}_\mathcal{T}(F(\mathcal{O}_{x,\xi}),F(\mathcal{O}_{x,\xi})) = k.
\]
Thus, if we prove that \(GF(\mathcal{O}_{x,\xi})\cong \mathcal{O}_{x,\xi}\) then
we can deduce that \eqref{eqn:bijectivity} is bijective.

\subsection{Reduction to injectivity of \eqref{eqn:bijectivity} for \(i=1\).} 
\label{red-inj-i1-sec}

Fix a generalized point \(\mathcal{O}_{x,\xi}\) and suppose that
the homomorphism in (\ref{eqn:bijectivity}) is injective for \(i=1\).

By Lemma \ref{lem:sheaf-lemma}, \(\QQ_{x,\xi}:=G(F(\mathcal{O}_{x,\xi}))\) is a
sheaf supported at $x$. Since the adjunction map is not
trivial, there is a surjection \(\delta\colon \QQ_{x,\xi}\to
\mathcal{O}_{x,\xi}\).  Indeed, it is not zero and \(\xi\) is irreducible, so it
is surjective. We need to show that \(\delta\) is bijective. There is a short exact
sequence
\[
  0\to \mathrm{Ker}(\delta)\to
  \QQ_{x,\xi}\xrightarrow{\delta}\mathcal{O}_{x,\xi}\to 0
\]
where \(\mathrm{Ker}(\delta)\) is supported at $x$ as well.

To see \(\mathrm{Ker}(\delta)=0\), it suffices to show
\(\mathrm{Hom}(\mathrm{Ker}(\delta),\mathcal{O}_{x,\eta})=0\) for any $\eta\in
\mathrm{Irr}(\mathrm{Aut}(x))$.  But we have the identification
\[
  \mathrm{Hom}(\mathrm{Ker}(\delta),\mathcal{O}_{x,\eta}) =
  \mathrm{Ker}\bigl(\mathrm{Hom}(\mathcal{O}_{x,\xi},\mathcal{O}_{x,\eta}[1])\to
  \mathrm{Hom}(\QQ_{x,\xi},\mathcal{O}_{x,\eta}[1])\bigr).
\]
Thus, injectivity of \eqref{eqn:bijectivity} for $i=1$ implies that
$\ker(\delta)=0$, i.e., $\QQ_{x,\xi}\simeq\OO_{x,\xi}$.

\subsection{Injectivity of \eqref{eqn:bijectivity} for \(i=1\) follows from
generic injectivity for \(i=1\)}

Note that due to decomposition \eqref{Ox-decomposition}, injectivity of \eqref{eqn:bijectivity} for \(i=1\) and all $x=y$ and $\xi$ is equivalent to injectivity
of 
\begin{equation}
  \mathrm{Ext}^1(\mathcal{O}_x,\mathcal{O}_x)\to
  \mathrm{Hom}^1_{\mathcal{T}}(F(\mathcal{O}_x),F(\mathcal{O}_x))
  \label{eqn:injectivity-i=1}
\end{equation}
for all closed points $x$.

By assumption, \(G\circ F\) is the Fourier-Mukai functor $\Phi_{\mathcal{Q}}$ given by some kernel
\(\mathcal{Q}\). For any point \(x:\mathrm{Spec}(k)\to
\XX\), the pullback \((x\times\id)^\ast \mathcal{Q}\) is exactly 
\[
    \QQ_x:=G\circ F(\OO_x)=\bigoplus_{\xi} \xi^\vee\ot\QQ_{x,\xi},
\]
so it is a sheaf. Hence, by Lemma \ref{lem:flatness-criterion}, $\QQ$ is flat
over \(\mathcal{X}\) (with respect to the first projection). Let
\(\varepsilon\colon \mathcal{Q}\to \Delta_\ast\mathcal{O_X}\) be the counit of adjunction
(recall that it is defined on the level of Fourier-Mukai kernels, see \cite[App.\ A]{caldararu-willerton}). This map is in fact surjective since $(x\times\id)^*(\varepsilon)$ is the
surjective map 
$$\QQ_x=\bigoplus_\xi \xi^\vee\ot \QQ_{x,\xi}\to\bigoplus_\xi \xi^\vee\ot \OO_{x,\xi}=\OO_x$$
given by the direct sum of the maps $\de:\QQ_{x,\xi}\to \OO_{x,\xi}$ which were shown to be surjective in \ref{red-inj-i1-sec}.
Thus, we have an exact sequence of coherent sheaves on $\XX\times \XX$
\[
0\to \mathcal{K}\to \mathcal{Q}\to \Delta_\ast\mathcal{O_X}\to 0.
\]
It follows that \(\mathcal{K}\) is flat over \(\mathcal{X}\) (via the first projection). 

If we assume injectivity of \eqref{eqn:injectivity-i=1} for a generic $x\in \XX$,
then as above we deduce that for generic point $x$, one has \((x\times\id)^*\mathcal{K}=0\). 
Since \(\mathcal{K}\) is flat over $\XX$, it follows that $\KK=0$, and 
the adjunction morphism $\varepsilon$ is an isomorphism.

\subsection{Generic injectivity for \(i=1\)}

We want to prove that for generic $x\in\XX$, the natural map \eqref{eqn:injectivity-i=1} 
is injective. This is equivalent to injectivity of the natural map
\begin{equation}\label{Ext1-main-map}
\Ext^1(\OO_x,\OO_x)\to \Ext^1(\QQ_x,\QQ_x).
\end{equation}

Recall that by \cite[Thm.\ 4.4]{kresch-geometrydm}, $\mathcal{X}$ is a quotient stack. 
Hence, by \cite[Prop.\ 5.2]{kresch-geometrydm}, there is a Zariski open substack \(\mathcal{Y\subset
X}\) of the form \(\mathcal{Y}\cong [Y/G]\), where \(Y\) is a quasi-projective
variety and \(G\) is a finite group. Let \(N\) be the kernel of the action and
\(H = G/N\). By shrinking \(Y\), we can assume \(H\) acts freely. Set the
quotient map to be \(\pi_Y\colon Y\to \bar{Y} = Y/H\). Denote also by \(\QQ\)
the sheaf \(\QQ\) restricted to \(\mathcal{Y\times Y}\).

By Theorem \ref{thm:bgll}, there is an equivalence of categories between
\(\mathfrak{Coh}(\mathcal{Y})\) and \(\alpha\)-twisted sheaves on
\(\bar{Y}_N=(\mathrm{Irr}(N)\times Y)/H\), where \(\alpha\) is the corresponding Brauer
class.  Let \(\QQ'\) be the twisted sheaf on $\bar{Y}_N\times\bar{Y}_N$ corresponding to \(\QQ\) under the corresponding equivalence
for the product. By Lemma \ref{lem:flatness-bgll}, \(\QQ'\) is still flat over
\(\pi_1\). Let \(\pi\colon U\to \bar{Y}_N\) be an \'etale cover trivializing \(\alpha\).
Then the twisted sheaf \((\pi\times\id)^*\QQ'\) on \(U\times \bar{Y}_N\)
satisfies the conditions of Lemma \ref{lem:bridgeland-hilbert-argument}.
Hence, we have an isomorphism
$$(\pi\times\pi)^*\QQ'\simeq (f\times \id_U)^*\OO_{\Xi}$$
(after possibly shrinking $U$), for some finite morphism $f:U\to \mathrm{Hilb}_{\ell}(U)$.
From this we get an isomorphism of functors (see e.g., \cite[Ex.\ 5.12]{huybrechts-fm})
\begin{equation}\label{FM-isom-eq}
\pi^*\circ \Phi_{\QQ'}\circ \pi_*\simeq \Phi_{\OO_{\Xi}}\circ f_*.
\end{equation}

Since $f$ is finite, the induced map on tangent spaces,
$$f_*:\Ext^1(\OO_u,\OO_u)\simeq T_uU\to \Ext^1(\OO_{f(u)},\OO_{f(u)})\simeq T_{f(u)}\mathrm{Hilb}_{\ell}(U),$$
is an isomorphism, hence, injective, for generic $u\in U$. Furthermore, since $\Xi$ is the universal family, the map
$$\Phi_{\OO_{\Xi}}:\Ext^1(\OO_z,\OO_z)\simeq T_z\mathrm{Hilb}_{\ell}(U)\to \Ext^1(\Xi_y,\Xi_y)$$
is injective for every $z\in\mathrm{Hilb}_{\ell}(U)$.
Hence, the composition $\Phi_{\OO_{\Xi}}\circ f_*$ induces an injective map from $\Ext^1(\OO_u,\OO_u)$, for generic $u\in U$.

Using isomorphism \eqref{FM-isom-eq}, we get that the map
$$\pi^*\circ\Phi_{\QQ'}\circ\pi_*:\Ext^1(\OO_u,\OO_u)\to \Ext^1(\pi^*\Phi_{\QQ'}\pi_*\OO_u,\pi^*\Phi_{\QQ'}\pi_*\OO_u)$$
is injective for generic $u\in U$. Now we observe that $\pi_*\OO_u=\OO_{\pi(u)}$ and the map
$$\pi_*:\Ext^1(\OO_u,\OO_u)\to \Ext^1(\OO_{\pi(u)},\OO_{\pi(u)})$$
is an isomorphism since $\pi$ is \'etale. Thus, we deduce that $\pi^*\circ\Phi_{\QQ'}$ induces an injective map from 
$\Ext^1(\OO_{\ov{y}},\OO_{\ov{y}})$ for generic $\ov{y}\in\bar{Y}_N$. Hence, the same is true for the map induced by
$\Phi_{\QQ'}$.

Note that we can run the above argument for every connected component of $\bar{Y}_N$, so that the injectivity of
the map induced by $\Phi_{\QQ'}$ holds for a dense open subset of points $\ov{y}\in\bar{Y}_N$.
It remains to recall that by Lemma \ref{BGLL-points-lem}, the structure sheaves of generalized points $(y,\xi)$ on $\YY$ correspond to
the sheaves $\OO_{\ov{(\xi,y)}}$ on $\bar{Y}_N$. This gives the required injectivity of \eqref{Ext1-main-map} for
generic $x\in\YY$.

\bibliographystyle{amsalpha}
\bibliography{bob-ffcriterion}
\end{document}